\documentclass[11pt,reqno]{article}
\usepackage{amsmath,amsthm,amssymb}
\usepackage{graphicx}
\usepackage{marginnote}

\def\R{\mathbb{R}}
\def\C{\mathbb{C}}
\def\P{\mathbb{P}}
\def\E{\mathbb{E}}

\def\Z{\mathbb{Z}}

\DeclareMathOperator{\var}{Var}

\DeclareMathOperator{\area}{Area}

\newtheorem{theorem}{Theorem}[section]
\newtheorem{proposition}[theorem]{Proposition}
\newtheorem{lemma}[theorem]{Lemma}
\newtheorem{claim}[theorem]{Claim}

\newtheorem{corollary}[theorem]{Corollary}

\def\eps{\varepsilon}

\def\eqd{\,{\buildrel d \over =}\,}

\title{Double roots of random Littlewood polynomials}
\author{Ron Peled\thanks{School of Mathematical Sciences, Tel Aviv University, Tel Aviv, Israel. E-mail: {\tt peledron@post.tau.ac.il}. Supported by an ISF grant and an IRG grant.}\and Arnab Sen\thanks{Department of Mathematics, University of Minnesota, USA. Email: {\tt arnab@umn.edu}. Supported by NSF grant 1406247.} \and Ofer Zeitouni\thanks{Faculty of Mathematics, Weizmann Institute
    of Science, Israel
    and Courant Institute, New York University, USA.
    Email: {\tt ofer.zeitouni@weizmann.ac.il}. Supported by an ISF grant and by
the Herman P. Taubman professorial chair of Mathematics at WIS.}}
\date{September 4, 2014. Revised January 13, 2015.}
\begin{document}

\maketitle
\begin{abstract}
We consider
random polynomials
whose coefficients are independent and uniform on $\{-1,1\}$. We
prove that the probability that such a polynomial of degree $n$ has
a double root is  $o(n^{-2})$ when $n+1$ is not divisible by $4$ and
asymptotic to $\frac{8\sqrt{3}}{\pi n^2}$ otherwise. This result is
a corollary of a more general theorem that we prove concerning
random polynomials with independent, identically distributed
coefficients having a distribution which is supported on $\{ -1, 0,
1\}$ and whose largest atom is strictly less than $1/\sqrt{3}$.  In
this general case, we prove that the probability of having a double
root equals the probability that either $-1$, $0$ or $1$ are double
roots up to an $o(n^{-2})$ factor and we find the asymptotics of the
latter probability.
\end{abstract}

\section{Introduction}
A \emph{Littlewood polynomial} is a polynomial whose coefficients
are all in $\{-1,1\}$. By a random Littlewood polynomial of degree
$n$ we mean a Littlewood polynomial chosen uniformly among all the
$2^{n+1}$ Littlewood polynomials of degree $n$. In this paper we
investigate the probability that a random Littlewood polynomial has
a double root and show that it is $O(n^{-2})$, and compute it up to an error of
order $o(n^{-2})$.

Our result concerning random Littlewood polynomials is a corollary of
a more general theorem that we now state.
Let $(\xi_j)$, $j\ge 0$, be an independent, identically distributed
sequence of random variables taking values in $\{-1,0,1\}$. Let
$n\ge 1$ and define the random polynomial $P$ by
\begin{equation*}
  P(z):=\sum_{j=0}^n \xi_j z^j.
\end{equation*}
For a
complex number $z$ define the event
\begin{equation*}
  D_z := \{z\text{ is a double root of }P\}.
\end{equation*}
\begin{theorem}\label{thm:double_root}
  If
  \begin{equation}\label{eq:coefficient_condition}
   \max_{x\in\{-1,0,1\}} \P(\xi_0 = x) < \frac{1}{\sqrt{3}}
  \end{equation}
  then
  \begin{equation}\label{eq:thm_conclusion}
    \P(P\text{ has a double root}) = \P(\cup_z D_z)=\P(D_{-1}\cup D_0\cup D_1) +
    o(n^{-2})\quad\text{as $n\to\infty$}.
  \end{equation}
\end{theorem}
Thus, up to a $o(n^{-2})$ factor, the probability of having a double
root is dominated by the probability that either $-1,0$ or $1$ are
double roots. Here and later in the paper we write $o(a_n)$ to
denote a term $\delta_n$, where the sequence $(\delta_n)$ depends
only on the distribution of $\xi_0$ and satisfies $\lim_{n\to\infty}
\delta_n/a_n = 0$. Similarly, $\delta_n=O(a_n)$ means that
$\limsup_{n\to\infty} |\delta_n|/a_n<\infty$.

Our next theorem calculates the asymptotics of the double root
probability.
\begin{theorem}\label{thm:double_root_asymptotics}
    Assume condition~\eqref{eq:coefficient_condition}. First,
  \begin{equation}\label{eq:double_root_prob_limit}
     \lim_{n\to\infty} \P(P\text{ has a double root}) = \P(\xi_0 = 0)^2.
  \end{equation}
  Second, if
  \begin{equation}\label{eq:no_atom_at_0}
    \P(\xi_0=0)=0
  \end{equation}
  then
  \begin{equation}\label{eq:double_root_prob_asymptotics}
    \P(P\text{ has a double root}) = \frac{L_n}{n^2} + o(n^{-2})
    \quad\text{as $n\to\infty$},
  \end{equation}
  where $L_n$ denotes the periodic sequence
  \begin{equation}\label{eq:limit_constant_value}
    L_n := \begin{cases}
        \frac{8\sqrt{3}}{\pi \var(\xi_0)}& \text{if $\E(\xi_0) = 0$
             and $n+1$ is divisible by $4$},\\
        \frac{4\sqrt{3}}{\pi \var(\xi_0)}&\text{if
        $\E(\xi_0) \neq 0$  and $n+1$ is divisible by
      $4$},\\
      0&\text{if $n+1$ is not divisible by $4$}.
    \end{cases}
  \end{equation}
\end{theorem}

We make a few remarks regarding the theorems.
\begin{enumerate}
    \item  The event that $P$ possesses a double root is the same as the
        event that $P$ and $P'$ have a common root, which necessarily
    must lie in the annulus ${\cal A}=\{1/2\leq |z|\leq 2\}$ or at $0$.
    Since the
    correlation coefficient between $P(z)$ and $P'(z)$ is
    bounded away from $1$
    as $n\to \infty$ uniformly in ${\cal A}$,
    a natural heuristic is that
        the probability that $P$ possesses a double root
        is up to a multiplicative constant
    asymptotically the same as the probability that $P$ and an
        independent copy of $P'$ possess a common root, which by local
        CLT considerations and some analysis should be at most
    of order
        $n^{-2}$
    when $\P(\xi_0=0)=0$ (in case one considers $P$  and
    an independent  copy $\tilde{P}$ of $P$, such an analysis was
        carried out in \cite{KZ13}). Directly carrying out this heuristic seems, however, challenging.
\item  We do not know if condition \eqref{eq:coefficient_condition},
or a condition of a similar kind, is necessary for the conclusion
\eqref{eq:thm_conclusion} of Theorem
    \ref{thm:double_root} to hold; the theorem does cover the
    interesting cases where the distribution of the
    coefficients $\xi_i$ is uniform on
  all three of $\{-1,0,1\}$ or uniform on any two of these values.
  See the open problems section for further information.

  \item In case $n+1$ is not divisible by $4$, our results for $\P(\xi_0=0)=0$
    do not yield the
      leading term in the asymptotic expansion of the left side
      of \eqref{eq:thm_conclusion}; by parity consideration,
      in that situation,
      $\pm1$ cannot be a
      double root of $P$. In that situation, one needs to
      consider also roots of unity  of algebraic degree larger than $1$. The
      asymptotics
      then depend on further arithmetic properties of $n$. While
      our methods could in principle be adapted to yield such results,
      we do not attempt to do so.
      We note, however, that under certain number theoretic
  assumptions, there exist infinitely many $n$ for which the
  polynomial $P$ is deterministically irreducible, indeed, even the
  deterministic polynomial $P\bmod 2$ is irreducible mod $2$, see
  \cite{MO09}.

  \item Our methods could also be used in evaluating the
           probability that $P$ possesses a root of multiplicity $k$. We
           expect that under the condition
           \eqref{eq:coefficient_condition}, the probability of
           having a root of multiplicity $k \ge 2$ (fixed) equals
           $\P(\text{either $-1, 0$ or $+1$ is a root of order of } k)
           + o(n^{-k^{2}/2})$.
           We have, however, not verified the details
           of this assertion.
           Note that, as described in the next remark, it is known
           \cite{FL99} that the probability that $1$ is a root of
           multiplicity $k$ is of order $O(n^{-k^2/2})$ for random Littlewood
           polynomials.
  \item When dealing with random Littlewood polynomials and when $n+1$ is divisible by $4$, the asymptotic
      probability that $-1$ or $1$ are double roots of $P$ is already known
      and has an interesting history which we briefly sketch. It
      suffices, as one may check simply (see
      \eqref{eq:1_and_minus_1_equality_in_distribution} and
      \eqref{eq:simultaneous_double_root}), to show that
      \begin{equation}
          \label{eq-of1}
     \P(P(1)=P'(1)=0) = \frac{4\sqrt{3}}{\pi n^2} +
  o(n^{-2}).\end{equation}
  That is, one needs to count the number of $\pm 1$
  sequences $\{a_i\}_{i=0}^n$
  such that $\sum_{i=0}^n a_i=0$ and $\sum_{i=1}^{n}  i a_i=0$.
  Setting $b_i=a_{i-1}$, this is the same as counting the number of
  solutions of the system of equations $\sum_{i=1}^{n+1} b_i=0$ and
  $\sum_{i=1}^{n+1} ib_i=0$, with $b_i\in \{-1,1\}$.
  The latter is a quantity appearing in coding theory, namely, the number
  of spectral-null codes of second order and length $n+1$,
  denoted ${\cal S}(n+1,2)$, which was
  evaluated (non rigorously, and with a slightly different motivation) already
  in \cite{SR}, and rigorously in \cite{FL99}.
  Both derivations start from the substitution
  $X_i=(b_i+1)/2$ to show that
  ${\cal S}(n+1,2)$ equals the number of partitions with distinct parts
  of $(n+1)(n+2)/4$
  into $(n+1)/2$ parts with largest part at most $n+1$.
  The authors in \cite{FL99} then derive  a local CLT, which implies the
  required asymptotics.
  Our proof proceeds
  with a somewhat different approach to the local CLT, using some ideas from
  \cite{KLP13}.
  \end{enumerate}

\subsection{Overview of the proof of Theorem \ref{thm:double_root}}
Recall that the minimal polynomial of an algebraic integer $\alpha$
is the monic polynomial  in $\mathbb{Z}[x]$ of least degree such
that $\alpha$ is a root of that polynomial. We denote by
$\deg(\alpha)$ the algebraic degree of an algebraic integer
$\alpha$, i.e., the degree of its minimal polynomial.

The first and perhaps most crucial step of our argument  is the following lemma which allows us to discard the algebraic integers with sufficiently high degrees.
The proof of the lemma is based on an
idea appearing in a work of Filaseta and Konyagin \cite{FK96}.
\begin{lemma}\label{lem:high_degree}(High degree)
  Under the assumption \eqref{eq:coefficient_condition} there exist
  constants $C,c>0$ such that for any $1\le d\le n$,
  \begin{equation*}
    \P(P\text{ has a double root $\alpha$ with $\deg(\alpha)\ge
    d$})\le C\exp(-cd).
  \end{equation*}
\end{lemma}

As we are aiming for an error of size $o(n^{-2})$, as in
\eqref{eq:thm_conclusion}, the lemma allows us to restrict attention
to algebraic integers $\alpha$ with $\deg(\alpha) = O(\log n)$. We
shall
then make use of Dobrowolski's result on Lehmer's conjecture
\cite{D79}  to further restrict attention to two cases: the case
when $\alpha$ is a root of unity or $\alpha  = 0$ and the case when
there is a conjugate $\beta$ of $\alpha$ such that $\beta$ lies a
bit far away
 from the unit circle, more precisely
 $|\beta| > 1 + \frac{c}{\log n} \big(\frac{\log\log n}{\log n}\big)^3$.
The first case is addressed in the following lemma whose proof relies on a classical anti-concentration result of S\'ark\"ozi and
Szemer\'edi \cite{SS65}.
\begin{lemma}\label{lem:root_of_unity}(Roots of unity)
  Under the assumption \eqref{eq:coefficient_condition} there exists
  a constant $C>0$ such that if $\alpha$ satisfies $\alpha^k = 1$ for some $k\ge 1$ then
  \begin{equation*}
    \P(\alpha\text{ is a root of $P'$})\le
    \left(\frac{C}{\lfloor\frac{n}{k}\rfloor}\right)^{\frac{3\deg(\alpha)}{2}}.
  \end{equation*}
\end{lemma}
Using Lemma~\ref{lem:root_of_unity}, we will show that if $\alpha$
is a root of unity with $\deg(\alpha) = O(\log n)$ and $\deg(\alpha)
\ge 2$, then $ \P(\alpha\text{ is a root of $P'$}) \le
O\left(\left(\frac{\log n\log\log \log n}{n}\right)^{3}\right)$.
Since there are not many such roots of unity, in fact $O\big((\log n
\log \log \log n)^2\big)$ of them, a simple union bound implies that
\[  \P( P' \text{ has a root $\alpha$ such that $\alpha$ is a root of unity  and } 2 \le \deg(\alpha) = O(\log n)) = o(n^{-2}).\]
Finally, we deal with the second case as follows. We will show that
the probability that $\alpha$ is a root of $P$ decreases very
rapidly with the distance of $\alpha$ from the unit circle.
\begin{lemma}\label{lem:off_circle}(Far from the unit circle)
  Under the assumption \eqref{eq:coefficient_condition}, for any algebraic integer $\alpha\neq 0$,
  \begin{equation*}
    \P(\alpha\text{ is a root of $P$})\le e^{-\frac{n\log 3}{2\lceil \log 3 /|\log
  |\alpha||\rceil}}.
  \end{equation*}
\end{lemma}
The proof of the above lemma is elementary and is based on a
sparsification argument. We shall apply the lemma for $\alpha$
satisfying $|\alpha| > 1 + \frac{c}{\log n} \big(\frac{\log\log
n}{\log n}\big)^3$. Since there are only $\exp(O((\log n)^2))$
potential roots of $P$ with algebraic degree $O(\log n)$ (see
Lemma~\ref{lem:number_of_minimal_polynomials}), a union bound yields
an error estimate of $o(n^{-2})$ for the second case too. Therefore,
we conclude that the probability that $P$ has a double root is the
same as the probability that $P$ has a double root at $-1, 0$ or $1$
up to an error of $o(n^{-2})$.

\subsection{Structure of the paper}
Section \ref{sec:high_alg_degree} is dedicated to handling roots of
high algebraic degrees, i.e., to the proof of Lemma
\ref{lem:high_degree}. Section \ref{sec:roots_unity} treats roots of
unity and provides the proof of Lemma \ref{lem:root_of_unity}.
Section \ref{sec:far_roots} handles roots that are far away from the
unit circle, providing the proof of Lemma \ref{lem:off_circle}.
Section \ref{sec:double_root} is dedicated to the deduction of
Theorem~\ref{thm:double_root}. Section \ref{sec:A_double_root} is
dedicated to the local CLT and proof of Theorem
\ref{thm:double_root_asymptotics}. The paper ends with a few open
questions.
\subsection{Acknowledgements}
We are grateful to Van Vu for mentioning the relevance of Lehmer's conjecture and Hoi H. Nguyen for the reference to Filaseta and
Konyagin \cite{FK96}.  Also, AS is indebted to Manjunath Krishnapur for many insightful discussions.

\section{High algebraic degree}
\label{sec:high_alg_degree}

In this section we prove Lemma~\ref{lem:high_degree}. We start with
two preliminary claims.

\begin{claim}\label{clm:jensen}
There exists a constant $M$ such that for any $n \ge 1$ and any
non-zero polynomial $f$ of the form $f(z) = \sum_{i=0}^n a_i z^i$
with $a_i \in \{-1,0,1\}$ for all $0 \le i \le n$,  the number of
$z\in\C$ for which $f(z) = 0$ and $|z|\ge\frac{3}{2}$ is at most
$M$.
\end{claim}
\begin{proof}
Assume, without loss of generality, that  $|a_n| = 1$. Let $\tilde
f(z) = z^n f(z^{-1}) = \sum_{i=0}^n a_i z^{n-i}$ be the reciprocal
polynomial of $f$. Denote by $N(f)$ the number of $z\in\C$ for which
$f(z) = 0$ and $|z|\ge\frac{3}{2}$. Then $N(f)$ is also the number
of $z\in\C$ for which $\tilde f(z)=0$ and $|z|\le \frac{2}{3}$.
Noting that $|\tilde{f}(0)|=1$ we may apply Jensen's formula (see,
e.g., \cite[Chapter 5.3.1]{A78}) and obtain for any $r> \frac{2}{3}$
that
\begin{equation*}
  \max_{0\le \theta\le 2\pi} \log|\tilde{f}(re^{i\theta})| \ge
  \frac{1}{2\pi}\int_0^{2\pi} \log|\tilde{f}(re^{i\theta})|d\theta =
  \log|\tilde{f}(0)| + \sum_{z\colon \tilde{f}(z)=0,\, |z|\le r} \log\left(\frac{r}{|z|}\right) \ge N(f)
  \log\left(\frac{r}{2/3}\right).
\end{equation*}
Observe that when $r<1$ we have $|\tilde{f}(re^{i\theta})|\le
\frac{1}{1-r}$ for all $\theta$. Thus
\begin{equation*}
  N(f)\le \frac{1}{(1-r)\log(3r/2)},\quad \frac{2}{3}<r<1
\end{equation*}
and substituting $r=0.82$, say, yields that $N(f)\le 26$, finishing
the proof.
\end{proof}
\begin{claim}\label{clm:integer_divisibility}
  Let $P$ be the random polynomial from
  Theorem~\ref{thm:double_root} and assume
  \eqref{eq:coefficient_condition}. There exist constants $C, c>0$
  such that for any $B>0$ we have
  \begin{equation*}
    \P\left(\text{P(3) is divisible by $k^2$ for some integer $k\ge B$}\right)\le C
    B^{-c}.
  \end{equation*}
\end{claim}
\begin{proof}
  Let $k\ge 1$ be an integer and let $r$ be the integer satisfying
  $3^r\le k^2<3^{r+1}$. By conditioning on $\xi_r, \xi_{r+1}, \ldots, \xi_n$ we have
\begin{align}\label{eq:P3ad}
\P\left(P(3) \bmod k^2 = 0\right) \le \max_{m\in \mathbb Z} \P\left(
\sum_{j=0}^{r-1} \xi_j 3^j \bmod k^2 = m\right) = \max_{m \in
\mathbb Z} \P\left(\sum_{j=0}^{r-1} \xi_j 3^j  = m\right),
\end{align}
where the last equality follows from the fact that
$\left|\sum_{j=0}^{r-1}\xi_j 3^j\right|\le\frac{1}{2}(3^r-1)$
deterministically and $k^2\ge 3^r$ by the definition of $r$. Write
$\max_{x\in\{-1,0,1\}} \P(\xi_0 = x) = \frac{1}{\sqrt{3}} -
\delta\in\left[\frac{1}{3},\frac{1}{\sqrt{3}}\right)$ by the
assumption \eqref{eq:coefficient_condition}. Observe that
\begin{equation}\label{eq:ternary_expansion}
  \text{the mapping $(a_0,\ldots, a_{r-1})\mapsto \sum_{j=0}^{r-1}
  a_j 3^j$ is one-to-one on $\{-1,0,1\}^r$}
\end{equation}
as the ternary expansion of an integer is unique. Thus,
\begin{equation}\label{eq:P3bd}
 \max_{m \in
\mathbb Z} \P\left(\sum_{j=0}^{r-1} \xi_j 3^j  = m\right) \le
\left(\frac{1}{\sqrt{3}} - \delta\right)^r.
 \end{equation}
Combining \eqref{eq:P3ad} and \eqref{eq:P3bd} with the fact that
$r>\frac{2\log k}{\log3} - 1$ we deduce that
\[  \P\left(P(3) \bmod k^2 = 0\right)  \le 3\left(\frac{1}{\sqrt{3}} - \delta\right)^{\tfrac{2 \log k}{\log 3}} = 3k^{-\gamma},\]
where $\gamma :=  - \log\left(\frac{1}{\sqrt{3}} - \delta\right) /
\log \sqrt 3
>1$. Summing over $k \ge B$ we obtain
\begin{equation*}
  \P\left(\text{P(3) is divisible by $k^2$ for some integer $k\ge B$}\right)  \le C B^{-(\gamma -
  1)},
  \end{equation*}
for some suitable constant $C>0$, as required.
\end{proof}

We now complete the proof of Lemma~\ref{lem:high_degree}.
Let $P$ be the random polynomial from
  Theorem~\ref{thm:double_root} and assume
  \eqref{eq:coefficient_condition}. Fix $1 \le d \le n$. Let
  $\alpha$ be an algebraic integer of degree $d$ with (monic) minimal
  polynomial $g$. Denote by $C(\alpha)$ the set of
  algebraic conjugates of $\alpha$ (i.e., the set of roots of
  $g$). Suppose that $\alpha$ is a double root of $P$.
  Then, necessarily $g$ divides $P$ and therefore
  $|\beta|\le 2$ for all $\beta\in C(\alpha)$ and, by
Claim~\ref{clm:jensen}, all but at most $M$ of the $\beta\in
C(\alpha)$ satisfy
$|\beta|\le \frac{3}{2}$. We conclude that
\begin{equation*}
  |g(3)| = \prod_{\beta\in C(\alpha)} |3-\beta| \ge (\tfrac{3}{2})^{d - M}  \ge c_1 (\tfrac{3}{2})^{d},
\end{equation*}
where $c_1 := (\tfrac{3}{2})^{-M}>0$. In addition, the facts that
$\alpha$ is a double root of $P$  and that  $\alpha$ cannot be a
multiple root of $g$, since in that case $\alpha$ is also a root of
$g'$ violating the minimality of $g$, imply that $g^2$ divides $P$
in $\mathbb{Z}[x]$ so that, in particular, the integer $P(3)$ is
divisible by $g(3)^2$. Putting the above facts together we arrive at
the inclusion of events
\begin{equation*}
  \left\{\text{$\alpha$ is a double root of $P$}\right\}\subseteq\left\{\text{$P(3)$
  is divisible by $k^2$ for some integer $k\ge
  c_1(\tfrac{3}{2})^d$}\right\}.
\end{equation*}
Lemma~\ref{lem:high_degree} now follows from
Claim~\ref{clm:integer_divisibility}.

\section{Roots of unity}
\label{sec:roots_unity}
In this section we prove Lemma~\ref{lem:root_of_unity}. We make use
of the following anti-concentration result of S\'ark\"ozi and
Szemer\'edi \cite{SS65}.
\begin{theorem}\label{thm:Sarkozy_Szemeredi}
  Let $(\eps_j)$, $1\le j\le N$, be independent random variables
  with $\P(\eps_j = 0) = \P(\eps_j = 1) = \frac{1}{2}$. There exists
  a constant $C>0$ such that for any \emph{distinct} integers $(a_j)$, $1\le j\le N$, we have
  \begin{equation*}
    \max_{m\in\Z} \P\left(\sum_{j=1}^N \eps_j a_j = m\right) \le
    \frac{C}{N^{3/2}}.
  \end{equation*}
\end{theorem}
Clearly, by a linear change of variable, the theorem continues to
hold when $\P(\eps_j = a) = \P(\eps_j = b) = \frac{1}{2}$ for any
$\{a,b\}\subset\Z$. The following corollary extends this to our
non-symmetric setting.
\begin{corollary}\label{cor:Sarkozy_Szemeredi}
  Let $(\xi_j)$ be as in Theorem~\ref{thm:double_root}. There exists a constant $C>0$
  such that for any \emph{distinct} integers $(a_j)$, $1\le j\le N$, we have
  \begin{equation*}
    \max_{m\in\Z} \P\left(\sum_{j=1}^N \xi_j a_j = m\right) \le
    \frac{C}{N^{3/2}}.
  \end{equation*}
\end{corollary}
\begin{proof}
  Using the assumption \eqref{eq:coefficient_condition} there exists
  some $p> \left(1 - \frac{1}{\sqrt{3}}\right)$,
  $a\neq b\in \{-1,0,1\}$ and a random variable
  $\eta$ supported in $\{-1,0,1\}$ such that if we let $\eps$ be uniform on $\{a,b\}$
  then the distribution of $\xi_1$ has the distribution of the mixture obtained by
  sampling $\eps$ with probability $p$ and sampling $\eta$ with
  probability $1-p$. Let $(\eps_j)$, $(\eta_j)$, $j\ge 1$, be
  independent identically distributed sequences with the
  distributions of $\eps$ and $\eta$ respectively.
  Independently, couple each
  $\xi_j$ to $(\eps_j, \eta_j)$ as the above mixture. Let $T$ be the
  random set of indices $1\le j\le N$ in which we sampled $\eps_j$
  to obtain $\xi_j$. Thus $T$ is distributed as a binomial with
  $N$ trials and success probability $p$. Using
  Theorem~\ref{thm:Sarkozy_Szemeredi} we have
  \begin{align*}
    \max_{m\in\Z} \P\left(\sum_{j=1}^N \xi_j a_j = m\right) &= \max_{m\in\Z}
    \E\P\left(\sum_{j=1}^N \xi_j a_j = m\,\Big|\,T, (\xi_j)_{j\notin T}\right)\le \E\max_{m\in\Z}\P\left(\sum_{j\in T} \xi_j
    a_j = m\,\Big|\,T\right) =\\
    &= \E\max_{m\in\Z}\P\left(\sum_{j\in T} \eps_j a_j =
    m\,\Big|\,T\right) \le \E\left(\min\left(1,\frac{C}{|T|^{3/2}}\right)\right).
  \end{align*}
    It remains to note that by standard concentration estimates for
  binomial random variables there exists some universal constant
  $c>0$ for which $\P(|T| \le \frac{1}{2} Np) \le \exp(-cNp)$. Thus
  \begin{equation*}
    \E\left(\min\left(1,\frac{C}{|T|^{3/2}}\right)\right) \le
    \exp(-cNp) + \frac{C}{\left(\frac{1}{2}Np\right)^{3/2}}.\qedhere
  \end{equation*}
\end{proof}

We complete now the proof of Lemma \ref{lem:root_of_unity}.
 Let $\alpha$ be such that $\alpha^k = 1$ for some $k\ge 1$.
 Observe that necessarily $\deg(\alpha)\le k$. Set
\begin{align*}
  J&:=\{j\colon 1\le j\le n\text{ and }0\le (j -1) \bmod k\le
  \deg(\alpha) -1\},\\
  \bar{J}&:=\{1,\ldots, n\}\setminus J.
\end{align*}
Define the random variables $(S_r)$, $0\le r\le\deg(\alpha) - 1$, by
\begin{equation*}
  S_r := \sum_{\substack{j\in J\\
  j-1\bmod k = r}} \xi_j j \alpha^{j-1} = \alpha^r \sum_{\substack{j\in J\\
  j - 1\bmod k = r}} \xi_j j
\end{equation*}
and
\begin{equation*}
  \bar{S} := \sum_{j\in \bar{J}} \xi_j j
  \alpha^{j-1}.
\end{equation*}
Observe that
\begin{equation*}
  P'(\alpha)=\sum_{j=1}^n \xi_j j \alpha^{j-1} = \sum_{j\in J} \xi_j j
  \alpha^{j-1} + \sum_{j\in \bar{J}} \xi_j j
  \alpha^{j-1} = \sum_{r=0}^{\deg(\alpha)-1} S_r + \bar{S}.
\end{equation*}
Now, by definition, $(S_r)$, $0\le r\le\deg(\alpha)-1$, are
independent and also independent of $\bar{S}$. In addition,
$(\alpha^r)$, $0\le r\le \deg(\alpha)-1$, are linearly independent
over the rational numbers, and therefore the equation
$\sum_{i=0}^{\deg(\alpha)-1} a_i \alpha^i=z$
has at most one integral solution $(a_0,\ldots,a_{\deg(\alpha)-1})$
for
a given
$z\in \C$. Hence
\begin{align*}
  \P(P'(\alpha)=0) &= \E\P\left(\sum_{r=0}^{\deg(\alpha)-1} S_r=-\bar{S}\,\Big|\,\bar{S}\right)
  \le \max_{z\in\C} \P\left(\sum_{r=0}^{\deg(\alpha)-1} S_r = z\right) =\\
  &= \prod_{r=0}^{\deg(\alpha)-1}\max_{z\in\C} \P(S_r = z) =
  \prod_{r=0}^{\deg(\alpha)-1}\max_{m\in\Z} \P\Bigg(\sum_{\substack{j\in J\\
  j-1\bmod k = r}} \xi_j  j = m\Bigg).
\end{align*}
Applying Corollary~\ref{cor:Sarkozy_Szemeredi}
and the fact that $|J|\geq \lfloor n/k\rfloor$ we conclude that
\begin{equation*}
  \P(P'(\alpha)=0) \le
  \left(\frac{C}{\lfloor\frac{n}{k}\rfloor}\right)^{\frac{3\deg(\alpha)}{2}}.
\end{equation*}

\section{Roots off the unit circle}
\label{sec:far_roots} In this section we prove
Lemma~\ref{lem:off_circle}. Let $\alpha\neq 0$ be an algebraic
integer. We assume $|\alpha|\neq 1$ as otherwise the lemma is
trivial. We note also that the probability that $\alpha$ is a root
of $P$ is the same as the probability that $1/\alpha$ is a root of
$P$ since $P(\alpha)$ has the same distribution as $\alpha^n
P(1/\alpha)$. Thus we assume without loss of generality that
$|\alpha|>1$. Define $j_0$ as the minimal positive integer for which
\begin{equation}\label{eq:j_0_prop}
  |\alpha|^{j_0} \ge 3.
\end{equation}
Write $P(z) = P_1(z) + P_2(z)$ with
\begin{equation*}
  P_1(z):=\sum_{k=0}^{\lfloor n/j_0\rfloor} \xi_{kj_0} z^{k
  j_0}\quad\text{and}\quad P_2(z):=P(z) - P_1(z).
\end{equation*}
The assumption \eqref{eq:j_0_prop} implies that the mapping
$T:\{-1,0,1\}^{\lfloor n/j_0\rfloor+1}\to\C$ defined by
\begin{equation*}
  (a_0, \ldots, a_{\lfloor n/j_0\rfloor})\mapsto \sum_{k=0}^{\lfloor n/j_0\rfloor} a_k \alpha^{k
  j_0}
\end{equation*}
is one-to-one (similarly to \eqref{eq:ternary_expansion}). Thus, as
$P_1(\alpha)$ and $P_2(\alpha)$ are independent,
\begin{align*}
  \P(\alpha\text{ is a root of $P$}) &= \E\left[\P(\alpha\text{ is a root of
  $P$}\,|\,P_2(\alpha))\right] =\\
  &= \E\left[\P(P_1(\alpha) =
  -P_2(\alpha)\,|\,P_2(\alpha))\right]\le
  \left(\max_{x\in\{-1,0,1\}}\P(\xi_0=x)\right)^{\lfloor
  n/j_0\rfloor+1}.
\end{align*}
Finally, assumption~\eqref{eq:coefficient_condition} and the
definition of $j_0$ imply that
\begin{equation*}
  \P(\alpha\text{ is a root of $P$})< \left(\frac{1}{\sqrt{3}}\right)^{\lfloor
  n/j_0\rfloor+1}\le e^{-\frac{n\log 3}{2j_0}} = e^{-\frac{n\log 3}{2\lceil \log 3 / \log
  |\alpha|\rceil}}.
\end{equation*}
\section{Probability of double root}
\label{sec:double_root}
In this section we prove Theorem~\ref{thm:double_root}.

By definition, any root of a monic polynomial with integer
coefficients is an algebraic integer. Thus, unless all coefficients
of $P$ are zero, the equation $P(z)=0$ is satisfied only by
algebraic integers $z$. We note this explicitly for later reference,
\begin{equation}\label{eq:double_root_zero_polynomial_estimate}
  \P\left(P\text{ has a root which is not an algebraic integer}\right) = \P(\xi_0 = 0)^{n+1} < 3^{-n/2}
\end{equation}
by assumption \eqref{eq:coefficient_condition}.

Let $c$ be the constant appearing in Lemma~\ref{lem:high_degree} and
note first that this lemma implies that
\begin{equation}\label{eq:double_root_high_degree_estimate}
  \P\left(P\text{ has a double root $\alpha$ with $\deg(\alpha)\ge \frac{3\log n}{c}$}\right)\le Cn^{-3}.
\end{equation}
Thus we may restrict attention to the following set of potential
roots,
\begin{equation*}
  A:=\left\{\alpha\in \C\colon \alpha\text{ is a root of a monic polynomial with coefficients in $\{-1,0,1\}$
  and $\deg(\alpha)<\frac{3\log n}{c}$}\right\}.
\end{equation*}
We now use use another argument of Filaseta and Konyagin
\cite{FK96} to bound the cardinality of  $A$.
\begin{lemma}\label{lem:number_of_minimal_polynomials}
  There exists a constant $C>0$ such that
  \begin{equation*}
    |A| \le C^{(\log n)^2}.
  \end{equation*}
\end{lemma}
\begin{proof}
  Let $\alpha\in A$ and denote by $C(\alpha)$ the set of its algebraic conjugates (including $\alpha$ itself). Since $\alpha$ is a root of a monic polynomial
  with coefficients in $\{-1,0,1\}$ it follows immediately that
  \begin{equation}\label{eq:conjugates_bound}
    |\beta|<2\;\;\text{ for each $\beta\in C(\alpha)$}.
  \end{equation}
  Now suppose $\deg(\alpha)=d$, let $g$ be the minimal polynomial of $\alpha$
  and denote
  \begin{equation*}
    g(z) = z^{d} + \sum_{j=0}^{d-1} a_j z^j = \prod_{\beta\in
    C(\alpha)} (z - \beta).
  \end{equation*}
  From this representation and \eqref{eq:conjugates_bound} we deduce
  that $|a_j|\le 4^{d}$ for each $j$, whence the integral vector $(a_0,\ldots,
  a_{d-1})$ has at most $4^{d^2}$ possibilities. We conclude that
  the number of $\alpha\in A$ with $\deg(\alpha) = d$ is at most
  $4^{d^2}$ and the lemma follows by summing over $d$.
\end{proof}
We continue by recalling the Mahler measure of an algebraic integer.
If $\alpha$ is an algebraic integer having minimal polynomial
\begin{equation*}
  g(z) = \prod_{\beta\in C(\alpha)}(z - \beta),
\end{equation*}
where $C(\alpha)$ is the set of algebraic conjugates of $\alpha$,
then the Mahler measure $M(\alpha)$ of $\alpha$ is
\begin{equation*}
  M(\alpha):=\prod_{\substack{\beta\in C(\alpha)\\|\beta|\ge 1}}
  |\beta|.
\end{equation*}
In particular, if $\alpha$ is an algebraic integer then
$M(\alpha)=1$ if and only if $\alpha=0$ or $|\beta|=1$ for all
$\beta\in C(\alpha)$.   Moreover, it follows from a classical theorem of
Kronecker \cite{K57} that if
$\alpha$ is an algebraic integer with $|\beta|=1$ for all $\beta\in
C(\alpha)$ then $\alpha$ is a root of unity. Finally, Lehmer's
conjecture \cite{L33} states that there exists some absolute constant $\mu>1$
such that
\begin{equation*}
 M(\alpha)=1\text{ or }M(\alpha)\ge\mu\text{ for all algebraic
  integers $\alpha$}.
\end{equation*}
We will make use of Dobrowolski's result on Lehmer's conjecture
\cite{D79} which says that
\begin{equation*}
  M(\alpha)=1\text{ or }\log(M(\alpha))\ge c'\left(\frac{\log\log(\deg(\alpha)+2)}{\log(\deg(\alpha)+2)}\right)^3\text{ for some $c'>0$ and all algebraic
  integers $\alpha$}.
\end{equation*}
We remark that earlier weaker results on Lehmer's conjecture such as those
of Blanksby and Montgomery \cite{BM71} or Stewart \cite{S78} would
also have sufficed for our purposes.

Now let $\alpha\in A$ and denote $d := \deg(\alpha)$. Assume that
$\alpha$ is neither $0$ nor a root of unity. It follows from the
preceding discussion that
\begin{equation*}
  \log(M(\alpha))\ge c'\left(\frac{\log\log(d+2)}{\log(d+2)}\right)^3
\end{equation*}
and hence, using that $e^x\geq 1+x$ for $x\geq 0$,
one concludes that
$\alpha$ has some algebraic conjugate $\beta$ satisfying
\begin{equation*}
  |\beta|\ge 1 +
  \frac{c'}{d}\left(\frac{\log\log(d+2)}{\log(d+2)}\right)^3.
\end{equation*}
Since $P(\alpha)=0$ if and only if $P(\beta)=0$ we may apply
Lemma~\ref{lem:off_circle} to deduce that
\begin{equation*}
  \P(\alpha\text{ is a root of $P$})\le e^{-\frac{n\log 3}{2\lceil \log 3 /\log
  |\beta|\rceil}}\le e^{-c''\frac{n(\log\log(d+2))^3}{d(\log(d+2))^3}}
\end{equation*}
for some $c''>0$. Putting this estimate together with
Lemma~\ref{lem:number_of_minimal_polynomials} and the fact that
$\deg(\alpha)<\frac{3\log n}{c}$ for $\alpha\in A$ yields that
\begin{equation}\label{eq:double_root_non_root_of_unity_estimate}
  \P\left(P\text{ has a root $\alpha\in A$ with $\alpha\neq 0$ and $\alpha$ not a root of
  unity}\right)\le C^{(\log n)^2}e^{-c'''\frac{n(\log\log \log n)^3}{\log n(\log
  \log n)^3}}
\end{equation}
for some $c'''>0$. It remains to consider the probability that $P$
has a root which is a root of unity. Let $\alpha$ be a root of unity
with $k$ being the minimal positive integer for which $\alpha^k = 1$
and $d := \deg(\alpha)$. By Lemma~\ref{lem:root_of_unity},
\begin{equation}\label{eq:root_of_unity_double_root_estimate}
  \P(\alpha\text{ is a double root of $P$})\le
    \left(\frac{C}{\lfloor\frac{n}{k}\rfloor}\right)^{\frac{3d}{2}}.
\end{equation}
Since the minimal polynomial of $\alpha$ is given by the $k$th cyclotomic polynomial $$\Phi_k(x):= \prod_{\substack{ 1 \le j \le k, \\ \gcd(j, k)=1}} (1- e^{ 2 \pi i j/k})$$ (see, for example, Lemma~7.6 and Theorem~7.7 of \cite{M96}), we have that
$d = \deg(\Phi_k) =  \varphi(k)$
where $\varphi$
is Euler's totient function, i.e., $\varphi(k)=|\{1\le j\le k\colon
\gcd(j,k) = 1\}|$. By standard estimates (see \cite[Theorem 2.9]{MV07}) there exists
some constant $c_1>0$ for which
\begin{equation*}
  d = \varphi(k) \ge \frac{c_1k}{\log\log (k+2)}.
\end{equation*}
Thus if $\alpha\in A$, so that in particular $\deg(\alpha) <
\frac{3\log n}{c}$, then
\begin{equation}\label{eq:root_of_unity_k_estimate}
  k \le C_1 \log n \log\log \log n
\end{equation}
for some $C_1>0$. Substituting back in
\eqref{eq:root_of_unity_double_root_estimate} yields
\begin{equation*}
  \P(\alpha\text{ is a double root of $P$})\le
    \left(\frac{C_2\log n\log\log \log n}{n}\right)^{\frac{3d}{2}}
\end{equation*}
for some $C_2>0$. In particular, since there are at most $k$ numbers
$\alpha$ for which $k$ is the minimal positive integer such that
$\alpha^k=1$ we conclude from the last two inequalities that
\begin{multline}\label{eq:double_root_root_of_unity_estimate}
  \P\left(P\text{ has a double root $\alpha\in A\setminus\{-1,1\}$ which
  is a root of
  unity}\right)\le\\
  \le (C_1\log n \log\log\log n)^2 \left(\frac{C_2\log n\log\log \log n}{n}\right)^{3} =
  o(n^{-2}).
\end{multline}
Theorem~\ref{thm:double_root} now follows by putting together
\eqref{eq:double_root_zero_polynomial_estimate},
\eqref{eq:double_root_high_degree_estimate},
\eqref{eq:double_root_non_root_of_unity_estimate} and
\eqref{eq:double_root_root_of_unity_estimate}.

\section{Asymptotics of the double root probability}
\label{sec:A_double_root} In this section we find asymptotics in
many cases for the probability that
the random polynomial $P$ has a
double root, proving Theorem~\ref{thm:double_root_asymptotics}.

We start with the proof of \eqref{eq:double_root_prob_limit}. By
Theorem~\ref{thm:double_root} we may focus on the probability that
either $-1, 0$ or $1$ are double roots of $P$. We have
\begin{equation}
   \P(0\text{ is a double root of $P$}) = \P(\xi_0 = 0)^2
\end{equation}
since $0$ is a double root of $P$ if and only if the free
coefficients of $P$ and $P'$ vanish. Thus,
\eqref{eq:double_root_prob_limit} follows by noting that the
probability that either $-1$ or $1$ are double roots of $P$ tends to
zero with $n$ by Lemma~\ref{lem:root_of_unity}.

In the rest of the section we assume \eqref{eq:no_atom_at_0} and
proceed to prove \eqref{eq:double_root_prob_asymptotics}. By
Theorem~\ref{thm:double_root} it suffices to find the asymptotics of
the probability that either $-1$ or $1$ are double roots of $P$.

We start with some simple observations. Note that
\begin{equation}\label{eq:parity_restrictions}
\begin{aligned}
  P(1) &\equiv P(-1) \equiv n+1 \bmod 2,\\
  P'(1) &\equiv P'(-1) \equiv \left\lceil\frac{n}{2}\right\rceil \bmod
  2.
\end{aligned}
\end{equation}
Thus,
\begin{equation}\label{eq:non-divisibility_conclusion}
  \P(-1\text{ or }1\text{ are double roots of $P$}) = 0\;\;\text{if
  $n+1$ is not divisible by $4$}.
\end{equation}
Together with Theorem~\ref{thm:double_root} this establishes the
case $L_n=0$ in \eqref{eq:double_root_prob_asymptotics} and
\eqref{eq:limit_constant_value}. We henceforth make the assumption
that
\begin{equation}\label{eq:divisibility_condition}
  \text{$n+1$ is divisible by $4$}.
\end{equation}
Next we note that $P(1)=0$ if and only if exactly half of the
$(\xi_j)_{0\le j\le n}$ are $1$. Thus, by standard large deviation
estimates for binomial random variables,
\begin{equation}\label{eq:non_zero_mean_estimate}
  \text{if $\E(\xi_0)\neq 0$ then }\P(P(1) = 0)\le C\exp(-cn)
\end{equation}
for some constants $C,c>0$. Additionally, it is straightforward to
check that
\begin{equation}\label{eq:1_and_minus_1_equality_in_distribution}
  \text{if $\E(\xi_0) = 0$ then $(P(1), P'(1)) \eqd (P(-1), P'(-1))$}.
\end{equation}
Lastly, since we have the equality of events
\begin{equation*}
  \{P'(1) = P'(-1) = 0\} = \left\{\sum_{k=1}^{\lceil\frac{n}{2}\rceil}
  (2k-1)\xi_{2k-1} = \sum_{k=1}^{\lfloor \frac{n}{2}\rfloor}
  2k\xi_{2k} = 0\right\},
\end{equation*}
it follows from Corollary~\ref{cor:Sarkozy_Szemeredi} that
\begin{equation}\label{eq:simultaneous_double_root}
  \P(P'(1) = P'(-1) = 0) = \P\left(\sum_{k=1}^{\lceil\frac{n}{2}\rceil}
  (2k-1)\xi_{2k-1}=0\right)\P\left(\sum_{k=1}^{\lfloor \frac{n}{2}\rfloor}
  2k\xi_{2k} = 0\right)\le \frac{C}{n^3}
\end{equation}
for some constant $C>0$. Putting together
Theorem~\ref{thm:double_root}, \eqref{eq:non_zero_mean_estimate},
\eqref{eq:1_and_minus_1_equality_in_distribution} and
\eqref{eq:simultaneous_double_root} we see that the remaining parts
of Theorem~\ref{thm:double_root_asymptotics} will follow by showing
that
\begin{equation}\label{eq:minus_1_double_root_asymptotics}
  \left|\P(-1\text{ is a double root of $P$}) - \frac{4\sqrt{3}}{\pi \var(\xi_0)n^2}\right| =
  o(n^{-2}).
\end{equation}
This asymptotics will be established via a local central limit
theorem. We rely on some ideas from \cite{KLP13}, but aim to give a
short proof tailored for our case rather than a general statement.

We wish to compare the probability distribution of $(P(-1), P'(-1))$
to the density of a Gaussian random vector with the same expectation
and covariance matrix. To this end we denote
\begin{equation}\label{eq:X_def}
  X:=(P(-1),P'(-1)) = \left(\sum_{j=0}^n \xi_j(-1)^j,\, \sum_{j=0}^n j\xi_j(-1)^{j-1}\right) = \sum_{j=0}^n (1,\, -j)\xi_j(-1)^j.
\end{equation}
A short calculation, using our standing assumption
\eqref{eq:divisibility_condition}, yields the expectation $\mu$ and
covariance matrix $\Sigma$ of $X$,
\begin{equation}\label{eq:mu_and_sigma_def}
\begin{aligned}
  &\mu=\left(0, \frac{n+1}{2}\E(\xi_0)\right),\\
  &\Sigma =\begin{pmatrix}
    \var(\xi_0)(n+1)&-\frac{\var(\xi_0)}{2}n(n+1)\\
    -\frac{\var(\xi_0)}{2}n(n+1)&\frac{\var(\xi_0)}{6}n(n+1)(2n+1)
  \end{pmatrix}.
\end{aligned}
\end{equation}
We also let $Y$ denote a Gaussian random vector in $\R^2$ having
expectation $\mu$ and covariance matrix $\Sigma$. By standard facts
regarding Gaussian random vectors, the characteristic function
$\hat{Y}:\R^2\to\C$ of $Y$ is
\begin{equation}\label{eq:Y_Fourier_transform}
  \hat{Y}(\theta)=\E e^{2\pi i\langle \theta, Y\rangle} = e^{2\pi i
  \langle \theta, \mu\rangle - 2\pi^2 \theta^t \Sigma \theta}
\end{equation}
and the density $f_Y:\R^2\to\R$ of $Y$ is
\begin{equation}\label{eq:Gaussian_density_with_Fourier transform}
  f_Y(y) = \frac{1}{2\pi\sqrt{\det(\Sigma)}}e^{-\frac{1}{2}(y -
  \mu)^t\Sigma^{-1}(y-\mu)} = \int_{\R^2} e^{-2\pi i
  \langle \theta, y\rangle}\hat{Y}(\theta)d\theta.
\end{equation}
The characteristic function $\hat{X}:\R^2\to\C$ of $X$ is also
simple to calculate, as $X$ is given in \eqref{eq:X_def} as a sum of
independent random vectors,
\begin{equation}\label{eq:Fourier_transform_of_X}
  \hat{X}(\theta)=\E e^{2\pi i \langle \theta, X\rangle} =
  \prod_{j=0}^{n}\left(p e^{2\pi i ((-1)^j\theta_1 + j(-1)^{j-1}\theta_2)}
  + (1-p) e^{-2\pi i ((-1)^j\theta_1 + j(-1)^{j-1}\theta_2)}\right).
\end{equation}
where we denote $\theta = (\theta_1, \theta_2)$ and let
\begin{equation*}
  p:=\P(\xi_0 = 1).
\end{equation*}
In addition, we note that by the parity restrictions
\eqref{eq:parity_restrictions} the values of $X$ lie in the lattice
$2\Z^2$ (again, using our standing assumption
\eqref{eq:divisibility_condition}). Therefore we have the
representation
\begin{equation}\label{eq:double_root_as_Fourier_transform_integral}
  \P(-1\text{ is a double root of $P$}) = \P(X = (0,0)) = 4\int_{\left[-\frac{1}{4},\frac{1}{4}\right]^2}
  \hat{X}(\theta)d\theta.
\end{equation}

The following proposition relates $\hat{X}$ to $\hat{Y}$ near zero
and shows that both are small away from zero.
\begin{proposition}
    \label{prop:6.2}
  Denote
  \begin{equation*}
    D:=\left[-n^{-5/12},
    n^{-5/12}\right]\times\left[-n^{-17/12},n^{-17/12}\right].
  \end{equation*}
  There exists an absolute constant $C>0$ and constants $C_p,c_p>0$ depending only on $p$ such that
  \begin{enumerate}
    \item\label{it:X_and_Y_near_zero} For every $\theta\in D$ we have $|\hat{X}(\theta) -
    \hat{Y}(\theta)|\le Cn^{-1/4}$.
    \item\label{it:X_far_from_zero} For every $\theta\in[-1/4,
    1/4]^2\setminus D$ we have $|\hat{X}(\theta)| \le
    C\exp\left(-c_pn^{1/6}\right)$.
    \item\label{it:Y_far_from_zero} $\int_{\R^2\setminus D}
    |\hat{Y}(\theta)|
    d\theta \le C_p\exp\left(-c_pn^{1/6}\right)$.
  \end{enumerate}
\end{proposition}
\begin{proof}
  We start with the proof of part~\ref{it:X_and_Y_near_zero}. Define a
  function $f:\R\to\C$ by
  \begin{equation*}
    f(x):=pe^{ix}+(1-p)e^{-ix}.
  \end{equation*}
  A simple calculation using the Taylor expansion of the logarithm
  (see \cite[Claim 4.10]{KLP13} for a similar claim) shows that for $0\le p\le 1$ and
  $|x|\le \frac{\pi}{4}$ we have
  \begin{equation*}
    f(x) = e^{(2p-1)ix - 2p(1-p)x^2+\delta(p,x)}
  \end{equation*}
  where $|\delta(p,x)| \le C' |x|^3$ for some absolute constant
  $C'>0$.  
  Plugging this into \eqref{eq:Fourier_transform_of_X} for
  $\theta\in D$ yields
  \begin{align*}
    \hat{X}(\theta) &= \exp\left(2\pi(2p-1) i\sum_{j=0}^n((-1)^j\theta_1 + j(-1)^{j-1}\theta_2) - 8\pi^2 p(1-p)\sum_{j=0}^n ((-1)^j\theta_1 + j(-1)^{j-1}\theta_2)^2 +
    \delta'\right)\\
    &=\exp\left(2\pi i\langle\theta, \mu\rangle -
    2\pi^2\theta^t\Sigma\theta +
    \delta'\right)=\hat{Y}(\theta)e^{\delta'}
  \end{align*}
  where the error term $\delta' = \sum_{j=0}^n \delta(p, 2\pi((-1)^j\theta_1 + j(-1)^{j-1}\theta_2))$ satisfies
  \begin{equation*}
    |\delta'|\le C''\sum_{j=0}^n |(-1)^j\theta_1 +
    j(-1)^{j-1}\theta_2|^3 \le C'''n^{-1/4}
  \end{equation*}
  and $C'',C'''>0$ denote absolute constants. This finishes the
  proof of part~\ref{it:X_and_Y_near_zero}.

  We now continue with the proof of part~\ref{it:Y_far_from_zero}. It is useful to proceed by finding a diagonal matrix which
  $\Sigma$ dominates. Since, for all $(\theta_1, \theta_2) \in\R^2$,
  \begin{equation*}
    n\theta_1\theta_2 =
    \left(\frac{\sqrt{7}}{2}\theta_1\right)\left(\frac{2n}{\sqrt{7}}\theta_2\right)\le\frac{1}{2}\left(\frac{7}{4}\theta_1^2
    + \frac{4n^2}{7}\theta_2^2\right) \le \frac{7}{8}\theta_1^2 +
    \frac{1}{7}n(2n+1)\theta_2^2
  \end{equation*}
  we conclude that
  \begin{equation*}
    \theta^t\Sigma\theta = \var(\xi_0)(n+1)\left(\theta_1^2 +
    \frac{1}{6}n(2n+1)\theta_2^2 - n\theta_1\theta_2\right) \ge
    \var(\xi_0)(n+1)\left(\frac{1}{8}\theta_1^2 + \frac{n(2n+1)}{42}
    \theta_2^2\right).
  \end{equation*}
  Thus, by \eqref{eq:Y_Fourier_transform}, we have
  \begin{equation*}
    \int_{\R^2\setminus D} |\hat{Y}(\theta)|d\theta = \int_{\R^2\setminus D} e^{-2\pi^2 \theta^t \Sigma
    \theta}d\theta\le \int_{\R^2\setminus D} e^{-2\pi^2\var(\xi_0)(n+1)\left(\frac{1}{8}\theta_1^2 + \frac{n(2n+1)}{42}
    \theta_2^2\right)}d\theta.
  \end{equation*}
  Now, letting $G_1, G_2$ be independent centered normal random variables with
  $\var(G_1) = \sigma_1^2:=\frac{2}{\pi^2\var(\xi_0)(n+1)}$ and $\var(G_2)
  = \sigma_2^2:=\frac{21}{2\pi^2\var(\xi_0)n(n+1)(2n+1)}$ we have that
  \begin{align*}
    \int_{\R^2\setminus D} |\hat{Y}(\theta)|d\theta \le
    2\pi\sigma_1\sigma_2\P((G_1,G_2)\notin D) &\le
    2\pi\sigma_1\sigma_2\left(\P(|G_1|> n^{-5/12}) + \P(|G_2| >
    n^{-17/12})\right)\\
    &\le \frac{C}{\var(\xi_0)^2}e^{-c\var(\xi_0)n^{1/6}}
  \end{align*}
  for some absolute constants $C,c>0$. This finishes the
  proof of part~\ref{it:Y_far_from_zero}.

  Finally we turn to part~\ref{it:X_far_from_zero}. By taking the constant $C$ sufficiently large we may assume that $n$ is large. Fix $\theta\in\left[-\frac{1}{4},\frac{1}{4}\right]^2$. Write
  \begin{equation*}
    x_j := 2((-1)^j\theta_1 + j(-1)^{j-1}\theta_2),\quad 0\le
    j\le n.
  \end{equation*}
  For a real number $x$, denote by $d(x,\Z)$ its distance to
  the nearest integer. Let
  \begin{equation*}
    J = J(\theta) :=\left\{0\le j\le n\colon d(x_j,\Z)\le
    \frac{1}{8}n^{-5/12}\right\}.
  \end{equation*}
  Using \eqref{eq:Fourier_transform_of_X}, if $|J|\le 9(n+1)/10$ then
  \begin{align*}
    |\hat{X}(\theta)| &=
  \prod_{j=0}^{n}\left|p e^{2\pi i x_j}
  + (1-p)\right| = \prod_{j=0}^n \sqrt{1 - 2p(1-p)(1-\cos(2\pi x_j))} \\
  &\le \left(1 - 2p(1-p)\left(1-\cos\left(\frac{\pi}{4} n^{-5/12}\right)\right)\right)^{\frac{n+1 - |J|}{2}} \le \left(1 - 20c_p n^{-5/6}\right)^{\frac{n+1}{20}}\le \exp\left(-c_pn^{1/6}\right).
  \end{align*}
  for some constant $c_p>0$ depending only on $p$.
  Hence it suffices to show that if
  \begin{equation}\label{eq:J_large}
    |J|\ge 9(n+1)/10
  \end{equation}
  then $\theta\in D$.

  Assume \eqref{eq:J_large}. We claim
  that there necessarily exist $j_1, j_2$ such that $j_1, j_2, j_1 +
  j_2\in J$. Indeed, we may take $j_1:=\min J\le \frac{n+1}{10}$ and
  we then have $J\cap(j_1 + J)\neq \emptyset$ by \eqref{eq:J_large} and the pigeonhole
  principle since both $J$ and $j_1 +
  J$ are contained in $\left[0,\frac{n+1}{10}+n\right]$.
  Thus,
   \begin{align}\label{eq:theta_1_integer_distance}
    d(2\theta_1, \Z) &= d((-1)^{j_1+j_2-1}x_{j_1+j_2} + (-1)^{j_1}x_{j_1} + (-1)^{j_2}x_{j_2}), \Z) \notag \\
    &\le
    d(x_{j_1+j_2}, \Z) + d(x_{j_1}, \Z) + d(x_{j_2}, \Z) \le \frac{3}{8}n^{-5/12},
  \end{align}
  whence, as $|\theta_1|\le \frac{1}{4}$,
  \begin{equation*}
    |\theta_1| = \frac{1}{2}d(2\theta_1, \Z) \le \frac{3}{16}n^{-5/12}.
  \end{equation*}
  Now, if $|\theta_2|\le n^{-17/12}$ then $\theta\in D$ and we are
  done. Assume, in order to obtain a contradiction, that
  $|\theta_2|>n^{-17/12}$.

  Let $I:=\{0\le j\le
  n\colon d(2j\theta_2, \Z)>\frac{1}{2}n^{-5/12}\}$. We claim that $|I|\ge n/3$. To see this let $k$ be the minimal positive
  integer for which $2k|\theta_2|>n^{-5/12}$. Since $|\theta_2|\le
1/4$ it follows that $2k|\theta_2|\le 1/2$. Thus,
  if $j\not\in I$ and $j\leq n-k$ then necessarily $j+k\in I$.
  In addition, $k\le \frac{1}{2n^{5/12}|\theta_2|}+1< n/2+1$.
In particular, $|I|\geq k$ which shows the claim when $k\geq n/3$.
Otherwise, assume $k<n/3$ and define $T:=\{j\in
[0,n-k]\cap\Z:\lfloor j/k\rfloor \, \mbox{\rm is even}\}$. We have
that $T$ and $T+k$ are disjoint subsets of $\{0,\ldots,n\}$ and for
each $j\in T$, either $j$ or $j+k$ belong to $I$. Hence $|I|\geq
|T|\geq (n-k+1)/2> n/3$, as claimed.


Now the assumption \eqref{eq:J_large} and the above claim imply that
  there exists some $j_3\in J$ for which $d(2j_3\theta_2, \Z)>
  \frac{1}{2}n^{-5/12}$, whence by \eqref{eq:theta_1_integer_distance},
  $d(x_{j_3}, \Z) > \frac{1}{8}n^{-5/12}$, contradicting the fact that $j_3\in J$.
\end{proof}

The asymptotics \eqref{eq:minus_1_double_root_asymptotics} are an
immediate consequence of Proposition \ref{prop:6.2}. Indeed, by
\eqref{eq:Gaussian_density_with_Fourier transform} and
\eqref{eq:double_root_as_Fourier_transform_integral}, and the proposition,
\begin{multline*}
  |\P(-1\text{ is a double  root of $P$}) - 4f_Y((0,0))| = 4\left|\int_{\left[-\frac{1}{4},\frac{1}{4}\right]^2}
  \hat{X}(\theta)d\theta - \int_{\R^2} \hat{Y}(\theta)d\theta\right| \\
  \le 4\left(\int_{D} |\hat{X}(\theta) - \hat{Y}(\theta)|d\theta + \int_{\left[-\frac{1}{4},\frac{1}{4}\right]^2\setminus D} |\hat{X}(\theta)|d\theta +
  \int_{\R^{2}\setminus D}
  |\hat{Y}(\theta)|d\theta\right) \\
  \le 4Cn^{-1/4}\area(D) + C\exp\left(-c_pn^{1/6}\right) +
  \frac{4C}{\var(\xi_0)^2}e^{-c\var(\xi_0)n^{1/6}} = o(n^{-2}).
\end{multline*}
In addition, by \eqref{eq:Gaussian_density_with_Fourier transform}
and \eqref{eq:mu_and_sigma_def} we have
\begin{align*}
  f_Y((0,0)) &=
  \frac{1}{2\pi\sqrt{\det(\Sigma)}}e^{-\frac{1}{2}\mu^t\Sigma^{-1}\mu}
  =
  \frac{\sqrt{12}}{2\pi\var(\xi_0)(n+1)\sqrt{n(n+2)}}e^{-\frac{3(n+1)^2}{2n(n^2+3n+2)}}
  =\\
  &=\frac{\sqrt{3}}{4\pi p(1-p)n^2} + o(n^{-2}).
\end{align*}
This finishes the proof of
\eqref{eq:minus_1_double_root_asymptotics} and completes the proof
of Theorem~\ref{thm:double_root_asymptotics}.
\qed

\section{Open questions}
We conclude the paper by listing down several open questions.
\begin{enumerate}
\item As mentioned in the introduction, we do not know if the assumption \eqref{eq:coefficient_condition} or any similar condition is
necessary for Theorem~\ref{thm:double_root} to hold. Recall that the
assumption enters into the proof mainly through
Claim \ref{clm:integer_divisibility} which, in turn, is used to
obtain the crucial Lemma~\ref{lem:high_degree}.

\noindent
{\bf Remark.} Mei--Chu Chang \cite{Ch} has kindly pointed out to the authors that for Claim \ref{clm:integer_divisibility}
to hold, in Assumption
\eqref{eq:coefficient_condition} the constant $1/\sqrt{3}=0.5774\ldots$  can be replaced by the supremum of $\rho$s so that there exists $q\in (1,\infty)$ such that $3^{(q-1)/2q}<\rho^q+(1-\rho)^q$, leading to the value $0.7615\ldots$. This still leaves open the question of whether \textit{any} assumption of the type  \eqref{eq:coefficient_condition}
is needed for Theorem~\ref{thm:double_root} to hold.

\item It is natural to try and extend Theorem~\ref{thm:double_root}
to more general coefficient distributions. This would require a
non-trivial modification of our approach as we relied in several
places on the fact that the potential roots of our random polynomial
are algebraic integers rather than the more general algebraic
numbers. A significant issue is to deal with potential roots of high
degree, providing an analogue of Lemma~\ref{lem:high_degree}.

\item  The following question does not involve any probability. Are there examples of  Littlewood polynomials with at least one non-cyclotomic double root? The same question had been asked by  Odlyzko and Poonen \cite{OP93}
    for polynomials with $0/1$ coefficients with the constant term equal
    to one. That question was later answered by Mossinghoff \cite{M03}
    who found  examples of several such polynomials
    with non-cyclotomic repeated roots.

\item  Another interesting question is
    to  bound  the probability that a random Littlewood polynomial is reducible. This is somewhat related to our original question regarding double roots - note that the probability of having a double root is dominated by the probability of being reducible. But handling irreducibility  seems to be much harder. To the best of our knowledge, it is open whether this probability goes to zero as $n$ increases. See
    the thread \cite{MO09}  for some partial results on this question.
\end{enumerate}

\noindent
{\bf Acknowledgments} We thank an anonymous referee for spotting an error in
our original proof of Proposition \ref{prop:6.2}, and   Mei--Chu Chang
for correspondence \cite{Ch} concerning Assumption
\eqref{eq:coefficient_condition}.

\end{document}